\documentclass{amsart}
\usepackage{graphicx} % Required for inserting images
\usepackage{mathrsfs}
\usepackage{amsfonts}
\usepackage{amsthm}
\usepackage{amsmath}
\usepackage{url}
\usepackage{amssymb}
\usepackage{hyperref}

\newtheorem{thm}{Theorem}[section]
\newtheorem{lem}[thm]{Lemma}
\newtheorem{prop}[thm]{Proposition}

\newtheorem{conj}{Conjecture}[section]

\theoremstyle{definition}
\newtheorem{defn}[thm]{Definition}

\DeclareMathOperator{\Alb}{Alb}
\DeclareMathOperator{\Gal}{Gal}
\DeclareMathOperator{\Min}{Min}
\DeclareMathOperator{\Sh}{Sh}

\newcommand{\R}{\mathbb{R}}
\newcommand{\univ}[1]{\tilde{#1}}

\title{A Real Shafarevich Conjecture for Universal Covers}

\author{Rodolfo Aguilar}
\address{\parbox{\linewidth}{Centro de Investigacion en Matemáticas, A.C. (CIMAT), Jalisco S/N, Col. Valenciana CP: 36023 Guanajuato, Gto, México}}
\email{\href{mailto:aaguilar.rodolfo@gmail.com}{aaguilar.rodolfo@gmail.com}}
\urladdr{\url{https://sites.google.com/view/rodolfo-aguilar/}}

\author{Cristhian Garay }
\address{\parbox{\linewidth}{Centro de Investigacion en Matemáticas, A.C. (CIMAT), Jalisco S/N, Col. Valenciana CP: 36023 Guanajuato, Gto, México}}

\email{cristhian.garay@cimat.mx}

\date{}

\begin{document}

\maketitle

\begin{abstract}
    The classical Shafarevich conjecture predicts that the universal cover of a complex smooth projective variety $X$ is holomorphically convex. In this paper, we propose a refinement of this conjecture for varieties defined over the reals. In order to do this, we introduce the notions of real holomorphic convexity and transverse holomorphic convexity to capture the geometric differences dictated by the real locus $X(\mathbb{R})$ of $X$. Specifically, we conjecture that the universal cover is real holomorphically convex when $X(\mathbb{R}) \neq \emptyset$, and dianalytic holomorphically convex when $X(\mathbb{R}) = \emptyset$. We prove this refined conjecture in two main cases: when $X$ is a curve, and when the fundamental group of $X$ is nilpotent.
\end{abstract}
\section{Introduction}

Let $X$ be a complex smooth projective variety. The classical Shafarevich conjecture states that its universal cover $\tilde{X}$ is holomorphically convex—that is, $\tilde{X}$ admits a proper holomorphic map to a Stein space. Recall that a Stein space is a complex analytic space that admits a continuous strictly plurisubharmonic proper exhaustion function $\psi$, see Section \ref{S:Prem} for details.  

Initial progress on the conjecture focused on surfaces with special structures \cite{GS85}. Then, the field shifted by restricting the structure of the fundamental group $\pi_1(X)$. This led to proofs for nilpotent \cite{K97}, reductive \cite{KR98,E04}, and ultimately linear fundamental groups \cite{EKPR12}. See the survey \cite{E11}. 

More recently, the focus shifted once again to allow $X$ to be normal quasi-projective, following the same development as the smooth projective case: for nilpotent groups \cite{GGK24,AC25}, for the reductive \cite{DYK23, B23} and finally the linear case \cite{BBT}. See also the recent survey \cite{D25}. It seems hard to go considerably further with the usual techniques.

From a different viewpoint, algebraic geometers have given a fresh view to classical problems by changing the field of definition of the variety, see \cite{K00}. Real algebraic geometry has been particularly fruitful recently, see \cite{BP24, FM25, BBG}. What we propose here is a refinement of the Shafarevich conjecture to the case when $X$ is a complex normal quasi-projective variety defined over the reals. For ease of exposition, we restrict our attention in this introduction to the smooth projective case; however, our results and methods extend naturally to normal quasi-projective varieties.

Let us consider a pair $(\tilde{X},\tilde{\sigma})$ consisting of  a complex manifold with an antiholomorphic involution $\tilde{\sigma}:\tilde{X}\xrightarrow[]{}\tilde{X}$. We say that $(\tilde{X},\tilde{\sigma})$ is real holomorphically convex if there exists a pair  $(\tilde{Y},\tilde{\sigma}_Y)$ consisting of a Stein space $\tilde{Y}$ together with an antiholomorphic involution $\tilde{\sigma}_Y$ and a proper equivariant holomorphic map $f:(\tilde{X},\tilde{\sigma}_X)\to (\tilde{Y}, \tilde{\sigma}_Y)$.

We can then state:

\begin{conj}\label{conj:Int1} Let $X$ be a complex smooth projective variety defined over $\mathbb{R}$ such that $X(\mathbb{R})\not = \varnothing$. Then its universal cover is real holomorphically convex.
\end{conj}

We use the real points of $X$ to lift the real structure of $X$ to a real structure on the universal cover, see Proposition \ref{prop:rscover}.

The classical strategy  to attack the Shafarevich conjecture --which is outlined in \cite{EKPR12}-- goes roughly as follows: first we have to construct the so-called Shafarevich variety $\Sh(X)$, see \cite{K95, C94} and the Shafarevich map $X\to \Sh(X)$. Then, show that the universal cover $\widetilde{\Sh(X)}$ is Stein. Therefore, to study Conjecture \ref{conj:Int1}, the strategy would be to check that at every stage of the above steps, the constructions give compatible real structures and equivariant proper maps with respect to these structure. 

Conjecture \ref{conj:Int1} follows easily for Riemann surfaces as a consequence of Riemann's uniformization theorem. See Proposition \ref{thm:CurvesNonEmpty}. When $X$ has nilpotent fundamental group, we follow the above strategy to show that the conjecture also holds, see Theorem \ref{thm:NilNonEmpty}.

In the case that $X(\mathbb{R})=\emptyset$, we cannot lift the real structure on $X$ to a real structure (an involution) on $\tilde{X}$: since there are no real basepoints to fix, we merely obtain a fixed-point-free antiholomorphic automorphism $\tilde{\sigma}_X$ whose square $\tilde{\sigma}_X^2$ is a non-trivial deck transformation. However, the Shafarevich variety is expected to admit a metric of non-positive curvature \cite{K95, DW25}. Viewing the universal cover as a CAT(0) space, the compactness of $X$ implies that the action of this automorphism $\tilde{\sigma}_X$ is cocompact and therefore semisimple (see Section \ref{ss:Metric geometry}). Consequently, $\tilde{\sigma}_X$ acts geometrically as an axial isometry.

This geometry guarantees the existence of a special totally real locus $L$ (the Min-Set), where the displacement function of $\tilde{\sigma}_X$ attains its global minimum. By the structure theorem for semisimple isometries, $L$ is a non-empty, totally geodesic subspace which splits isometrically as $Y \times \mathbb{R}$, and along which $\tilde{\sigma}_X$ acts as a translation. Because $L$ contains infinite geodesic lines (the translation axes), it is non-compact.

Therefore, any invariant plurisubharmonic exhaustion function $\psi$ constructed from the distance to $L$ will fail to be a global proper exhaustion function, as its sublevel sets contain the infinite translation orbits along $L$. This geometric obstruction appears even in the simplest case of Riemann surfaces.

Our solution to this is to pass to the quotient $\tilde{X}/\langle \tilde{\sigma}\rangle$, where $L$ has become compact. This necessitates leaving the category of complex manifolds and passing to dianalytic manifolds, manifolds where the transition functions are either holomorphic or antiholomorphic, see Section \ref{ss:dianalytic}. Moreover, it is not difficult to show that the definition of plurisubharmonicity using the Hessian criterion is invariant under pullbacks from antiholomorphic invertible maps, see Lemma \ref{lem:SteinInva}, which allows us to define the notion of a Stein dianalytic space. 

Thus, we will say then that a pair $(\tilde{X}, \tilde{\sigma})$ with $\tilde{X}$ is a complex manifold and $\tilde{\sigma}$ is a fixed-point-free antiholomorphic automorphism is \emph{dianalytic holomorphically convex} if we have an equivariant proper holomorphic map $(\tilde{X}, \tilde{\sigma})\to (\tilde{Y}, \tilde{\sigma}_Y)$ such that $\tilde{Y}/\langle  \tilde{\sigma}_Y \rangle$ is a Stein dianalytic space. With this language, we can state:

\begin{conj}\label{conj:Int2} Let $X$ be a smooth projective variety defined over the reals such that $X(\mathbb{R})=\emptyset$. Then the pair  $(\tilde{X}, \tilde{\sigma})$ is dianalytic holomorphically convex. 
\end{conj}

In Theorem \ref{thm:curvesEmpty}, we use Riemann's uniformization theorem to prove Conjecture \ref{conj:Int2} for $X$ of dimension one. In Theorem \ref{thm:NilNonEmpty}, we prove Conjecture \ref{conj:Int2} for varieties having nilpotent fundamental group.

\section*{Acknowledgement} 
The first author was supported by a postdoctoral fellowship, Estancias Posdoctorales por México 2025, from the Secretaría de Ciencia, Humanidades, Tecnología e Innovación (SECIHTI), Mexico.

\section{Preliminaries}\label{S:Prem}

\subsection{Real setting}
Many things inherent to complex algebraic and complex analytic geometry can be done at the level of the real numbers, but this is tricky. Generally speaking, a real analytic space (or algebraic variety) can be one of two things: 
\begin{enumerate}
    \item a scheme over $\mathbb{R}$ with some extra characteristics, or 
    \item a complex analytic space together with a particular involution
\end{enumerate}

For a technically accurate (but  complicated) theory of real analytic spaces seen as quotients of  complex analytic spaces (their complexification, considered in the category of locally ringed spaces), see \cite[\S1.5]{VB1}. Before taking this quotient, real analytic spaces can be defined as pairs $(X,\sigma)$ consisting of a complex analytic space $X=(X,\mathcal{O}_X)$ together with an anti-holomorphic involution $\sigma:X\xrightarrow[]{}X$. This includes real algebraic varieties, seen as schemes over $\mathbb{R}$. The real locus of the pair $(X,\sigma)$ is $X(\R) := \{x \in X \mid \sigma(x) = x\}$.

At the level of topology, there is no problem: the topological space is the quotient $X/\sigma$. The problem is how to define correctly the sheaves and the morphisms: the structure sheaf is $(\pi_*\mathcal{O}_X)^{\langle\sigma\rangle}$, where $\pi:X\xrightarrow[]{}X/\langle\sigma\rangle$ is the quotient map, and the morphisms $X/\langle\sigma_X\rangle\xrightarrow[]{}Y/\langle\sigma_Y\rangle$ are morphism of locally ringed spaces, but once again, these can be specified as equivariant morphisms $X\xrightarrow[]{}Y$. 

For curves, we are covered by \cite[\S3]{GH}. A real algebraic curve $Y=(Y,\mathcal{O}_Y)$ will be a complete, non-singular, geometrically connected curve of genus $g$ over $\mathbb{R}$, and we denote by $Y(\mathbb{R})$ the set of $\mathbb{R}$-rational points.

The  topology of these objects is well understood. It depends on the three discrete invariants $g,n,a$, where $0\leq n\leq g+1$ is the number of connected components of $Y(\mathbb{R})$, and $a\in\{0,1\}$.

\begin{prop}\label{prop:rscover}
    Let $X$ be a complex manifold equipped with a real structure, that is, an antiholomorphic involution $\sigma: X \to X$. Let $\pi: \univ{X} \to X$ be the universal covering map.  If the real locus $X(\R) = \{x \in X \mid \sigma(x) = x\}$ is non-empty,  then $\univ{X}$ admits a canonical real structure $\univ{\sigma}: \univ{X} \to \univ{X}$ such that $\pi \circ \univ{\sigma} = \sigma \circ \pi$.
\end{prop}

\begin{proof}
    Because the real locus is non-empty, we may fix a basepoint $x \in X(\R)$ such that $\sigma(x) = x$. Choose a point $\univ{x} \in \univ{X}$ in the fiber $\pi^{-1}(x)$. 
    
    Consider the continuous map $\sigma \circ \pi: \univ{X} \to X$. Because the domain $\univ{X}$ is simply connected, the universal property of covering spaces guarantees the existence of a unique continuous lift $\univ{\sigma}: \univ{X} \to \univ{X}$ such that $\pi \circ \univ{\sigma} = \sigma \circ \pi$ and $\univ{\sigma}(\univ{x}) = \univ{x}$.

    We first verify that $\univ{\sigma}$ is an involution. Consider the composition $\univ{\sigma}^2 = \univ{\sigma} \circ \univ{\sigma}: \univ{X} \to \univ{X}$. Since $\pi \circ \univ{\sigma}^2 = \sigma^2 \circ \pi = \pi$, the map $\univ{\sigma}^2$ is a lift of the identity map $\text{id}_X$. Furthermore, $\univ{\sigma}^2(\univ{x}) = \univ{\sigma}(\univ{x}) = \univ{x}$. Since the identity map $\text{id}_{\univ{X}}$ is also a lift of $\text{id}_X$ fixing the basepoint $\univ{x}$, the uniqueness of lifts implies that $\univ{\sigma}^2 = \text{id}_{\univ{X}}$.
    
    Next, we establish the regularity and antiholomorphicity of the lift. Because $\pi$ is a covering map between complex manifolds, it is a local biholomorphism. For any point $\univ{z} \in \univ{X}$, there exist open neighborhoods $U \ni \univ{z}$ and $V \ni \pi(\univ{z})$ such that $\pi|_U: U \to V$ is a biholomorphism. Let $\pi^{-1}_{loc}: V \to U$ denote its holomorphic inverse. 
    
    In a sufficiently small neighborhood of $\univ{z}$, we can express the lift as the composition:
    \[
        \univ{\sigma} = \pi^{-1}_{loc} \circ \sigma \circ \pi.
    \]
    The map $\pi$ is holomorphic, $\sigma$ is antiholomorphic, and $\pi^{-1}_{loc}$ is holomorphic. The composition of smooth maps is smooth, which proves that $\univ{\sigma}$ is smooth. Furthermore, the composition of two holomorphic maps and one antiholomorphic map is antiholomorphic. Therefore, $\univ{\sigma}$ is an antiholomorphic involution, defining a real structure on $\univ{X}$.
\end{proof}

It may happen that a pair $(X,\sigma_X)$  does not have real points, i.e. $X(\mathbb{R})=\varnothing$. In this case the involution $\sigma_X$ constructed in Proposition \ref{prop:rscover} lifts only to an antiholomorphic automorphism of $\tilde{X}$ whose square is a deck transformation.

\subsection{Plurisubharmonicity}\label{ss:psh}
All complex spaces are assumed to be reduced and paracompact.

\begin{defn} Let $X$ be a complex manifold. A function $\psi:X\to [-\infty, \infty)$ is called \emph{plurisubharmonic (psh)} if it satisfies the following conditions:
\begin{enumerate}
    \item $\psi$ is upper semi-continuous, 
    \item $\psi$ is not identically $-\infty$ on any connected component of $X$, 
    \item For every holomorphic map $\gamma:\Delta\to X$ from the unit disk $\Delta\subset \mathbb{C}$, the composition $\psi\circ\gamma:\Delta\to [-\infty, \infty)$ is subharmonic.
\end{enumerate}
\end{defn}

Let $(z_1,\ldots, z_n)$ be local holomorphic coordinates on $X$. If $\psi$ is of class $\mathcal{C}^2$, the \emph{complex Hessian matrix} of $\psi$ is the Hermitian matrix given by 
$$H(\psi)=\left( \frac{\partial^2 \psi}{\partial z_j \partial \bar{z}_k} \right)_{1\leq j,k\leq n}.  $$
The function $\psi$ is plurisubharmonic if and only if $H(\psi)$ is positive semi-definite at every point. It is called \emph{strictly} plurisubharmonic if $H(\psi)$ is strictly positive definite everywhere,  see \cite{FG}.

To define holomorphic convexity in complex spaces, we follow \cite{N62}.

\begin{defn}
    Let $V$ be a complex space. A function $\psi:V\to \mathbb{R}$ is (\emph{strictly}) \emph{plurisubharmonic} if for every point $p\in V$, there exists an open neighborhood $U\subset V$ of $p$, a closed holomorphic embedding $i:U\hookrightarrow W$ into an open set $W\subset \mathbb{C}^N$, and a smooth, (strictly) plurisubharmonic function $\Psi\in \mathcal{C}^\infty(W)$ such that $\psi|_U=\Psi \circ i$.
\end{defn}

By \cite[Theorem II]{N62}, a complex space is Stein if and only if it admits a continuous, strictly plurisubharmonic proper exhaustion function. Next, following Cartan \cite{C60} and \cite[Section 12]{BBT},  a complex analytic space is holomorphically convex if and only if it admits a proper holomorphic map to a Stein space.

\subsection{Real structures and Dianalytic spaces}\label{ss:dianalytic}
Here we give the analogous definitions of Section \ref{ss:psh} but taking into account the real structure. 
\begin{defn}\label{defn:realHolConv}
    Let $X$ be a complex analytic space equipped with an antiholomorphic involution $\sigma_X:X\to X$. We say that $X$ is \emph{real holomorphically convex} if it exists a Stein space $Y$ together with an antiholomorphic involution $\sigma_Y:Y\to Y$ and an equivariant proper holomorphic map $f:X\to Y$.  
\end{defn}

It may look at first sight that we could define a \emph{real Stein space} as a pair $(X,\sigma)$ with $X$ a complex Stein space and $\sigma$ an antiholomorphic involution such that its strictly plurisubharmonic proper exhaustion function is invariant under $\sigma$. However, it turns out that we can always construct such an invariant function using the following Lemma.

\begin{lem}\label{lem:SteinInva}
    Let $\sigma:X\to Y$ be an antiholomorphic diffeomorphism between complex manifolds, and let $\psi:Y\to\mathbb{R}$ be a $\mathcal{C}^2$ (strictly) plurisubharmonic function.  Then the pull-back by $u=\psi\circ \sigma$ is also a (strictly) plurisubharmonic function.
    \end{lem}
\begin{proof}
    Let $z=(z_1,\ldots, z_n)$ be local holomorphic coordinates on $X$ and $w=(w_1,\ldots, w_n)$ on $Y$. As $\sigma(z)=w$ is antiholomorphic, it satisfies $\frac{\partial w_j}{\partial z_k}=0$. By the complex chain rule, the complex Hessian of $u$ is given by: 
    $$\frac{\partial^2 u}{\partial z_k \partial \bar{z_l}}=\sum_{j,m} \frac{\partial^2 \psi}{\partial w_j \partial \bar{w}_m} \frac{\partial w_j}{\partial \bar{z_l}}\frac{\partial \bar{w}_m}{\partial z_k}. $$
    Since $\sigma$ is antiholomorphic, the term $\frac{\partial w_j}{\partial \bar{z}_l}$ is  the complex conjugate of $\frac{\partial \bar{w}_j }{\partial z_l}$. Let $J$ be the invertible Jacobian matrix of the conjugate map $\bar{\sigma}$, with entries $J_{mk}=\frac{\partial \bar{w}_m}{\partial z_k}$.
    We rewrite the above equation as matrices, using the notation of section \ref{ss:psh}:
    $$H(u)=J^T \overline{H(\psi)} \bar{ J}$$
    As $\psi$ is real-valued, its complex Hessian $H(\psi)$ is a Hermitian matrix, and its complex conjugate $\overline{H(\psi)}=H(\psi)^T$ is also Hermitian. We conclude by using linear algebra: the complex conjugate of a positive (semi-)definite Hermitian matrix is positive (semi-)definite and because $J$ is invertible, $H(u)$ is also positive (semi-)definite. 
\end{proof}

Using this, we can always construct an invariant psh function on a Stein manifold. We sketch an argument for analytic Stein spaces in the proof of Theorem \ref{thm:NilNonEmpty}.

\begin{defn}
    A \emph{dianalytic space} is a topological space $X$ equipped with an atlas of charts mapping into $\mathbb{C}^n$ such that all transition functions are either holomorphic or antiholomorphic.
\end{defn}
If $X$ has dimension one, it is also called a Klein surface.

\begin{defn}
    A \emph{dianalytic Stein space} is a pair $(Y,\sigma_Y)$ consisting of a complex space $Y$ and a properly discontinuous, fixed-point-free antiholomorphic automorphism $\sigma_Y$, such that $Y$ admits a continuous, strictly plurisubharmonic function that is invariant under $\sigma_Y$ and descends to a proper exhaustion function on the quotient space $Y/\langle \sigma_Y \rangle$.
\end{defn}

\begin{defn}
    Let $X$ be a complex analytic space equipped with a properly discontinuous fixed-point-free antiholomorphic automorphism $\sigma_X$. We say that $X$ is \emph{dianalytic holomorphically convex} with respect to $\sigma_X$, if it admits an equivariant proper holomorphic map $f:X\to Y$ to a dianalytic Stein space $(Y,\sigma_Y)$.
\end{defn}

\subsection{Real Shafarevich conjecture}
We are now ready to formulate our generalization of the Shafarevich conjecture to the real setting. 

\begin{conj}
    Let $X$ be a smooth projective variety defined over $\mathbb{R}$. Let $\tilde{X}$ be its universal cover and $\tilde{\sigma}_X$ a properly discontinuous lift of its real structure. If $X(\mathbb{R})\not = \emptyset$, then $\tilde{X}$ is real holomorphically convex. If $X(\mathbb{R})=\emptyset$, then $\tilde{X}$ is dianalytic holomorphically convex.
\end{conj}

\subsection{Metric geometry of semisimple isometries}\label{ss:Metric geometry}

We conclude this section with definitions characterizing the geometry of the lifted real structure when the variety lacks real points. In this setting, we rely on the geometry of non-positive curvature (CAT(0) spaces), which encompasses both the hyperbolic geometry of curves and the flat geometry of the Albanese covers.

\begin{defn}
    Let $(M,d)$ be a metric space and $f:M\to M$ an isometry. The \emph{displacement function} of $f$ is $\delta_f(x)=d(x,f(x))$. The \emph{Min-Set} of $f$ is the locus where this displacement is minimized:
    $$\text{Min}(f)=\left\lbrace x\in M \mid \delta_f(x) = \inf_{y\in M} \delta_f(y) \right\rbrace. $$
\end{defn}

An isometry is called \emph{semisimple} if the infimum of the displacement function is attained (i.e., $\text{Min}(f) \neq \emptyset$). If the infimum is not attained, the isometry is called \emph{parabolic}.

\begin{thm}[{\cite[6.2]{BH99}}]
    Let $M$ be a simply connected, complete geodesic metric space of non-positive curvature (a CAT(0) space). If $f$ is a semisimple isometry, then $\text{Min}(f)$ is a non-empty, closed, convex, totally geodesic subspace of $M$. Moreover, $\text{Min}(f)$ splits isometrically as a product $Y \times \mathbb{R}$, where $f$ acts as a translation on the $\mathbb{R}$ factor.
\end{thm}

In the context of smooth projective varieties, the action of the fundamental group on the universal cover is cocompact. By \cite[6.10]{BH99}, cocompact actions on CAT(0) spaces consist entirely of semisimple isometries. Consequently, the lifts of real structures in our setting always possess a non-empty Min-Set.

\begin{defn}
    A semisimple isometry $f$ is called \emph{axial} (or hyperbolic) if $\inf \delta_f > 0$. In the specific case where $f$ is an orientation-reversing isometry of a Riemannian manifold (such as a lifted real structure), we refer to $f$ as a \emph{glide reflection}. Geometrically, its action restricts to a translation along the axes in $\text{Min}(f)$ composed with an orthogonal reflection/rotation that preserves $\text{Min}(f)$.
\end{defn}

\section{Real curves}

\subsection{Non-empty real locus}

We begin by verifying Conjecture \ref{conj:Int1} for curves. In this dimension, the universal cover is already a Stein manifold ($\mathbb{H}$). While Lemma \ref{lem:SteinInva} guarantees the existence of some invariant exhaustion function by averaging, we show below that the canonical hyperbolic geometry provides a natural one.
\begin{prop}\label{thm:CurvesNonEmpty}
    Let $C$ be a smooth projective curve defined over the reals with $C(\mathbb{R}) \neq \emptyset$ and genus $g\geq 2$. Then its universal cover $\tilde{C}$ is real holomorphically convex.
\end{prop}

\begin{proof}
    Let $\sigma$ denote the antiholomorphic involution on $C$. Fix a basepoint $x\in C(\mathbb{R})$ to lift the real structure to an involution $\tilde{\sigma}$ on $\tilde{C}$. By the Uniformization Theorem, $\tilde{C}$ is biholomorphic to the unit disk $\Delta$ (or $\mathbb{H}$), which is a Stein manifold.
    
    To show that the pair $(\tilde{C}, \tilde{\sigma}) $ satisfies Definition \ref{defn:realHolConv}, we can take the target space $Y = \tilde{C}$, the map $f$ as the identity, and $\sigma_Y = \tilde{\sigma}$. We only need to exhibit a strictly plurisubharmonic proper exhaustion function invariant under $\tilde{\sigma}$.

    Equip $\tilde{C}$ with the Poincaré metric. Since this metric is unique, the antiholomorphic diffeomorphism $\tilde{\sigma}$ acts as an isometry. Because it fixes the lift $\tilde{x}$ of the real basepoint, $\tilde{\sigma}$ is a reflection (an elliptic isometry). Choosing coordinates where $\tilde{x}=0$ (in the disk model), $\tilde{\sigma}$ acts as complex conjugation $\tilde{\sigma}(z)=\bar{z}$.

    Consider the squared hyperbolic distance function $\psi(z) := d_{\mathbb{H}}^2(z, 0)$. Being a distance function in a negatively curved space, it is strictly convex and thus strictly subharmonic. Moreover, since $\tilde{\sigma}$ is an isometry fixing $0$, the function is invariant:
    \[ \psi(\tilde{\sigma}(z)) = d^2(\tilde{\sigma}(z), \tilde{\sigma}(0)) = d^2(z, 0) = \psi(z). \]
    
    Thus, $\psi$ serves as the required invariant exhaustion function.
\end{proof}

\subsection{Empty real locus}

\begin{thm}\label{thm:curvesEmpty}
    Let $C$ be a smooth projective curve defined over the reals with empty real locus $C(\mathbb{R})=\emptyset$ and genus $g\geq 2$. Then the pair $(\tilde{C}, \tilde{\sigma})$ is a dianalytic Stein space.
\end{thm}

\begin{proof}
    As $C(\mathbb{R})=\emptyset$, the antiholomorphic involution $\sigma$ is fixed-point-free. By the universal property of covering spaces, we lift $\sigma$ to an antiholomorphic map $\tilde{\sigma}:\tilde{C}\to \tilde{C}$. Because $\sigma$ is an involution, the square of the lift $\tilde{\sigma}^2$ acts as a deck transformation of $\tilde{C}$.
    
    By the Uniformization Theorem, $\tilde{C} \cong \mathbb{H}$. We equip $\tilde{C}$ with the unique hyperbolic metric of curvature $-1$. Since $\tilde{\sigma}$ is a diffeomorphism preserving the conformal structure (antiholomorphically), the uniqueness of the canonical metric implies that $\tilde{\sigma}$ acts as an orientation-reversing isometry.

    Since $C(\mathbb{R})=\emptyset$, the map $\tilde{\sigma}$ has no fixed points on $\mathbb{H}$. By the classification of isometries of the hyperbolic plane, a fixed-point-free orientation-reversing isometry is a \emph{glide reflection}.
    
    Consequently, $\tilde{\sigma}$ preserves a unique geodesic axis $L\subset \mathbb{H}$. The action decomposes into a reflection across $L$ and a translation along $L$ by distance $l>0$. Choosing coordinates where $L$ is the imaginary axis, the map takes the form:
    $$\tilde{\sigma}(z)=e^l(-\bar{z}). $$

    Consider the transversal exhaustion function defined by the squared hyperbolic distance to the axis $L$:
    $$\phi(z)=d_{\mathbb{H}}^2(z,L).$$
    Because $\tilde{\sigma}$ acts as an isometry preserving the set $L$, $\phi$ is strictly invariant:
    $$\phi(\tilde{\sigma}(z))=d_{\mathbb{H}}^2(\tilde{\sigma}(z),\tilde{\sigma}(L))=d_{\mathbb{H}}^2(z,L)=\phi(z). $$

    The squared distance to a geodesic in $\mathbb{H}$ is a strictly convex function, and therefore strictly subharmonic everywhere.
    
    Finally, we verify the properness on the quotient. The action of $\tilde{\sigma}$ on the axis $L\cong \mathbb{R}$ is a translation by $l$; thus, the quotient $L/\langle \tilde{\sigma} \rangle$ is a compact circle. Consequently, the sublevel sets of the descended function (which are tube neighborhoods of this circle) are compact. Thus, $\phi$ descends to a continuous, strictly subharmonic proper exhaustion function on the Klein surface $\tilde{C}/\langle \tilde{\sigma} \rangle$, proving that $(\tilde{C}, \tilde{\sigma})$ is a dianalytic Stein space.
\end{proof}

\section{Nilpotent covers}

If $G$ is a group, we denote its nilpotent completion by $G^{nilp}$.
If $G=\pi_1(X)$ with $X$ a topological space admiting a universal cover, we denote the cover associated to $G^{nilp}$ by $\tilde{X}^{nilp}$. 
\subsection{Non-empty real locus}
Let $X$ be a smooth projective variety defined over $\mathbb{R}$, and assume that $X(\mathbb{R})\not = 0$. Let $\sigma_X:X\to X$ denote the antiholomorphic involution associated to its real structure. Using a basepoint $x\in X(\mathbb{R})$ for the fundamental group $\pi_1(X,x)$, we lift the real structure of $X$ to an antiholomorphic involution $\tilde{\sigma}_X$ in its universal cover $\tilde{X}$ by Proposition \ref{prop:rscover}.

\begin{thm}\label{thm:NilNonEmpty}
    The Malcev universal nilpotent cover $\tilde{X}^{nilp}$ is real holomorphically convex.
\end{thm}
\begin{proof} The real structure on $X$ induces a real structure $\sigma_A$ of the Albanese variety $\Alb(X):=H^0(\Omega_X^1)^*/H_1(X,\mathbb{Z})$, see \cite[D.6]{FM}. We can see this directly by using the antilinear involution on $H^0(\Omega_X^1)$ given by $\omega \mapsto \overline{\sigma^* \omega}$ or by using the universal property of the Albanese variety and the holomorphic map $X \to \overline{X}$ (where $\overline{X}$ denotes the same differential variety $X$ but with its conjugate complex structure). With this real structure, the Albanese map $\alpha:X\to\Alb(X)$ is equivariant:
\begin{equation}
    \alpha\circ \sigma_X = \sigma_A \circ \alpha.
\end{equation}

From here, we follow the strategy of \cite[Thm 9]{AC25}, taking care of the involutions at every step.

Consider the Stein factorization of the Albanese map
$$X\overset{g}\to Z' \overset{h}{\to} \Alb(X), $$
where $g$ has connected fibers, $h$ is a finite morphism and $Z'$ is a normal projective variety. We want to show that $Z'$ inherits an antiholomorphic involution $\sigma_{Z'}$ making $g$ and $h$ equivariant with respect to the involutions $\sigma$ and $\sigma_A$ respectively. 
This follows from Galois descent. Since $X$ is defined over $\mathbb{R}$, we consider the morphism of schemes over $\mathbb{R}$. Stein factorization commutes with the strictly flat base change $\mathbb{C}/\mathbb{R}$ \cite[Tag 03GY]{SP}. Thus, $Z_\mathbb{C}'$ is canonically isomorphic to the base change of the Stein factorization over $\mathbb{R}$. This gives $Z'$ a natural, equivariant action of the Galois group $\Gal(\mathbb{C}/\mathbb{R})$, yielding the real structure $\sigma_{Z'}$.

Recall that we fixed  $x\in X(\mathbb{R})\not = \emptyset$, using the equivariance of $g$, we have that $z_0'=g(x_0)\in Z'(\mathbb{R})$. Together with the antiholomorphic involution $\sigma_{Z'}$, this induces a real structure $\tilde{\sigma}_{Z'}$ on the universal cover $\tilde{Z}'$ of $Z'$ by Proposition \ref{prop:rscover}.

We have by Theorem 9 of \cite{AC25} (see also Corollary 5) that 
\begin{equation}
    \pi_1(X)^{nilp} \cong \pi_1(Z')^{nilp}.
\end{equation}
In this compact case, the argument is easy and goes as follows. By \cite{ADH16} and because $Z'$ is normal, the nilpotent completions of $\pi_1(X)$ and $\pi_1(Z')$ are constructed using $H^1(Y)$ and the kernel of $H^1(Y)\otimes H^1(Y)\to H^2(Y)$ with $Y=X, Z'$ respectively. By properties of the Albanese map, we have that $H^1(X)=H^1(Z')$. The condition on $H^2(Z')$ follows from a strictness argument using \cite{D74}. See the proof of \cite[Thm 9]{AC25} for full details.

The diagonal real structure $\sigma_{prod}(x, \tilde{z}') = (\sigma_X(x), \tilde{\sigma}_{Z'}(\tilde{z}'))$ restricts to the fiber product precisely because $g \circ \sigma_X = \sigma_{Z'} \circ g$, and this restriction coincides with $\tilde{\sigma}_X$ by the uniqueness of basepoint-fixing lifts in covering space theory. This yields an equivariant proper holomorphic map:
\begin{equation}
\tilde{g}: \tilde{X}^{nilp} \to \tilde{Z}'^{nilp} \quad \text{such that} \quad \tilde{g} \circ \tilde{\sigma}_X = \tilde{\sigma}_{Z'} \circ \tilde{g},
\end{equation}
where these spaces denote the corresponding Malcev universal nilpotent covers.

By the definition of the Stein factorization, the map $h:Z'\to \Alb(X)$ is a finite morphism onto its image $\alpha(X)$. Therefore, the map on the nilpotent covers, $\tilde{h}:\tilde{Z}'^{nilp}\to \tilde{\alpha_X}(X)^{nilp}$, is also a finite covering.

As $\tilde{\alpha_X}(X)^{nilp}\subset \widetilde{\Alb}(X)^{nilp}\cong \mathbb{C}^N$ is a closed analytic subvariety, it follows that it is Stein. By \cite{GR79}, a finite holomorphic covering space of a Stein space is Stein. Thus, $\tilde{Z}'^{nilp}$ is a Stein space. Using definition \ref{defn:realHolConv}, we have proved the theorem. 

Let us give some details on how to generalize Lemma \ref{lem:SteinInva} to obtain an invariant psh function.

Now, by \cite{N62}, there exists a strictly plurisubharmonic exhaustion function $\psi':\tilde{Z}'^{nilp}\to \mathbb{R}$. It may be the case that $\psi'$ is not invariant under $\tilde{\sigma}_{Z'}$. 

Taking together the Albanese map $\tilde{h}: \tilde{Z}'^{nilp} \to \mathbb{C}^N$ with a finite set of holomorphic separating functions $G$ and their conjugate-symmetrized counterparts $H(x) = \overline{G(\tilde{\sigma}_{Z'}(x))}$, we construct a proper holomorphic embedding $\Psi = (\tilde{h}, G, H)$ into the ambient space $E = \mathbb{C}^N \times \mathbb{C}^K \times \mathbb{C}^K$. This embedding is strictly equivariant with respect to a standard linear antiholomorphic involution on $E$. Consequently, following Narasimhan's criteria for complex spaces, pulling back the strictly invariant, strictly plurisubharmonic standard Euclidean metric from $E$ directly furnishes the required invariant exhaustion function $\psi$ on the singular space $\tilde{Z}'^{nilp}$.
\end{proof}

\subsection{Empty real locus}

\begin{thm}\label{thm:NilEmpty} Let $X$ be a smooth projective variety defined over the reals. If  $X(\mathbb{R})=\emptyset$, then the Malcev universal nilpotent cover $\tilde{X}^{nilp}$ is dianalytic holomorphically convex.
\end{thm}
\begin{proof}
    Let $\alpha:X\to \Alb(X)$ be the Albanese map and write its Stein factorization $X\overset{g}{\to} Z' \overset{h}{\to} \Alb(X)$. As in the proof of Theorem \ref{thm:NilNonEmpty}, $Z'$ inherits a real structure $\sigma_{Z'}$. Choose any basepoint $x\in X$ and compatible lifts in the nilpotent covers. We obtain lifted maps $\tilde{g}, \tilde{h}$ and $\tilde{\sigma}_X, \tilde{\sigma}_{Z'}, \tilde{\sigma}_A$ which satisfy:
    $$\tilde{g}\circ \tilde{\sigma}_X = \tilde{\sigma}_{Z'}\circ \tilde{g}, \quad \tilde{h}\circ \tilde{\sigma}_{Z'} = \tilde{\sigma}_A \circ \tilde{h}. $$

    By \cite[Thm 9]{AC25}, we have $\pi_1(X)^{nilp}=\pi_1(Z')^{nilp}$, hence $\tilde{X}^{nilp}\cong X\times_{Z'} \tilde{Z}'^{nilp}$ is the fiber product,  $\tilde{g}$ is proper and holomorphic, and $\tilde{h}$ is finite onto its image inside $\mathbb{C}^{N}\cong \widetilde{\Alb}(X)^{nilp}$.

    The map $\tilde{\sigma}_A:\mathbb{C}^{N}\xrightarrow[]{}\mathbb{C}^{N}$ is a fixed-point-free antiholomorphic affine isometry, so we have that  $\tilde{\sigma}_A(z) = A\bar{z} + b$. The Min-Set $L = \Min(\tilde{\sigma}_A)$ is a non-empty affine subspace. Since $\tilde{\sigma}_A$ is antiholomorphic, $L$ is a totally real subspace of $\mathbb{C}^N$, i.e., $L \cap iL = \{0\}$.

    Consider the squared Euclidean distance to this affine subspace: $\Phi(z) = d_{\mathbb{C}^N}^2(z, L)$.
    
    In $\mathbb{C}^N$, the squared distance to a totally real subspace is a strictly plurisubharmonic function. This is because the Hessian of $\Phi$ is positive definite on the normal space to $L$. Since $L$ is totally real, the map $J: TL \to N_L$ is an isomorphism. Thus, strict convexity in normal directions implies strict plurisubharmonicity in all complex directions.

    Since $\tilde{\sigma}_A$ preserves $L$, we have that $\Phi$ is invariant. The quotient space $Y = \mathbb{C}^N / \langle \tilde{\sigma}_A \rangle$ is a complex manifold containing the compact totally real torus $L / \langle \tilde{\sigma}_A \rangle$. The function $\Phi$ descends to a strictly plurisubharmonic proper exhaustion function on $Y$.

\end{proof}

\bibliographystyle{smfalpha}
\bibliography{sample}
\end{document}